\documentclass[10pt,reqno]{amsart}

\usepackage{amssymb,latexsym,amsmath}
\usepackage{hyperref}

\def\Dom{\mathop{\rm Dom}}
\def\sup{\mathop{\rm sup}}
\def\Real{\mathop{\rm Re}}

\newtheorem{theorem}{Theorem}[section]

\newtheorem{corollary}[theorem]{Corollary}
\newtheorem{remark}[theorem]{Remark}
\begin{document}

\title[]{Stability of solutions to abstract evolution equations in Banach spaces under nonclassical assumptions}

\author{N. S. Hoang}

\address{Mathematics Department, University of West Georgia,
Carrollton, GA 30118, USA}
\email{nhoang@westga.edu}

\subjclass[2000]{34G20, 37L05, 44J05, 47J35.}

\date{}

\keywords{Evolution equations, stability, Lyapunov stable, asymptotically stable}

\begin{abstract}
The stability of the solution to the equation $(*)\dot{u} = F(t,u)+f(t)$, $t\ge 0$, $u(0)=u_0$ is studied. 
Here $F(t,u)$ is a nonlinear operator in a Banach space $\mathcal{X}$ for any fixed $t\ge 0$ and $F(t,0)=0$, $\forall t\ge 0$. 
We assume that the Fr\'echet derivative of $F(t,u)$ is H\"{o}lder continuous of order $q>0$ with respect to $u$ for any fixed $t\ge 0$, i.e., $\|F'_u(t,w) - F'_u(t,v)\|\le \alpha(t)\|v - w\|^{q}$, $q>0$. 
We proved that the equilibrium solution $v=0$ to the equation $\dot{v} = F(t,v)$ is Lyapunov stable under persistently acting 
perturbation $f(t)$ if $\sup_{t\ge 0}\int_0^t \alpha(\xi)\|U(t,\xi)\|\, d\xi<\infty$ and $\sup_{t\ge 0}\|U(t)\|<\infty$. Here, $U(t):=U(t,0)$ and $U(t,\xi)$ is the solution to the equation $\frac{d}{dt}{U}(t,\xi) = F'_u(t,0)U(t,\xi)$, $t\ge \xi$, $U(\xi,\xi)=I$, where $I$ is the identity operator in $\mathcal{X}$. 
Sufficient conditions for the solution $u(t)$ to equation (*) to be bounded and for $\lim_{t\to\infty}u(t) = 0$ are proposed and justified. 
Stability of solutions to equations with unbounded operators in Hilbert spaces is also studied. 
\end{abstract}
\maketitle

\pagestyle{plain}

\section{Introduction}
Consider the following equation
\begin{equation}
\label{eq1}
\dot{u} = F(t,u) + f(t),\qquad t\ge 0,\quad u(0)= u_0,\quad \dot{u} := \frac{d u}{dt},
\end{equation}
in a Banach space $\mathcal{X}$. 
It is assumed in equation \eqref{eq1} that $F(t,u)$ is an operator function from $\mathbb{R}\times\mathcal{X}$ to $\mathcal{X}$, nonlinear in general, for any fixed $t\ge 0$,
and that
\begin{equation}
\label{ae5}
\|F(t,u) - F'_u(t,0)u\| \le \alpha(t) \|u\|^p,\qquad p>1,\quad t\ge 0,\quad u\in \mathcal{X},
\end{equation}
where $F'_u(t,0) :=F'_u(t,u)|_{u=0}$ and $F'_u(t,u)$ denotes the Fr\'echet derivative of $F(t,u)$ for any fixed $t\ge 0$. 
Here, $\|\cdot\|$ denotes the norm in $\mathcal{X}$. 
Assume also that $f(t)$ is a function on $\mathbb{R}_+:=[0,\infty)$ with values in $\mathcal{X}$ and
\begin{equation}
\label{ae6}
\|f(t)\| \le \beta(t),\qquad t\ge 0.
\end{equation}
Note that inequality \eqref{ae5} implies $F(t,0) = 0$. Thus, $u=0$ is an equilibrium solution to the equation 
$$
\dot{u} = F(t,u),\qquad t\ge 0.
$$
We assume that the functions $\alpha(t)$ and $\beta(t)$ in \eqref{ae5} and \eqref{ae6} are in $L^1_{loc}([0,\infty))$ and that $\alpha(t)$ and $\beta(t)$ are nonnegative on $[0,\infty)$.  
Also, we assume that equation \eqref{eq1} has a unique local solution. 
A stronger assumption on the local existence of a solution to equation \eqref{eq1} is made in Assumption A below. By a solution to problem \eqref{eq1} we mean a classical solution. 
Specifically, a global solution to \eqref{eq1} is a continuously differentiable function $u:[0,\infty)\to \mathcal{X}$ which satisfies equation \eqref{eq1}. A local solution to equation \eqref{eq1} 
is a continuously differentiable function $u:[0,T)\to \mathcal{X}$, for some $T>0$, which solves equation \eqref{eq1}. Thus, the solution space for the global existence is $C^1([0,\infty);\mathcal{X})$ and for the local existence is $C^1([0,T);\mathcal{X})$. 


Stability of the solution to equation \eqref{eq1} in special forms has been extensively studied in the literature (see, e.g., \cite{C}, \cite{CL}, \cite{DK}, \cite{D}, \cite{H}, \cite{L}, \cite{T}). 
A special case for equation \eqref{eq1} is the equation
\begin{equation}
\label{s1}
\dot{u} = Au,\qquad u(0) = u_0,\qquad \dot{u}:=\frac{du}{dt},
\end{equation}
where $A$ is an $n$-by-$n$ matrix and $u:[0,\infty)\to \mathbb{R}^n$. 
The classical stability result for equation \eqref{s1} states that 
if all eigenvalues of $A$ lie in the half-plane $\Real\lambda <0$, then the solution to equation \eqref{s1} exists globally, is unique, and is asymptotically stable. If $A$ has an eigenvalue which lies in the 
half-plane $\Real\lambda > 0$, then $u(t)$ is not bounded, in general. 

The following nonlinear differential equation was studied in \cite{DK}:
\begin{equation*}
\dot{u} = A(t)u + f(t),\qquad u(0) = u_0.
\end{equation*}
It was assumed in \cite{DK} that $A(t)$ is a linear and bounded operator in a Banach space $\mathcal{X}$ and that $f(t)$ is a function from $[0,\infty)$ to $\mathcal{X}$. Let $U(t,\xi)$ be the solution to
\begin{equation*}
\dot{U}(t,\xi) = A(t)U(t,\xi),\qquad t\ge \xi,\qquad U(\xi,\xi) = I,\qquad \dot{U}(t,\xi):=\frac{d}{dt}U(t,\xi),
\end{equation*}
where $I$ is the identity operator in $\mathcal{X}$. 
Define
\begin{equation*}
\kappa = \overline{\lim}_{t,s\to\infty} \frac{\ln \|U(t+s,s)\|}{t}. 
\end{equation*}
Then it is known that if $\kappa<0$, then the solution $u=0$ is asymptotically stable when $f=0$ (see, e.g., \cite{DK}).

In \cite{R2}, \cite{R3}, \cite{R4} and \cite{R5} the following equation was studied
\begin{equation}
\label{l1}
\dot{u} = A(t)u + F(t,u) + f(t),\qquad u(0) = u_0.
\end{equation}
It is assumed in \cite{R4} that $A(t)$ is a linear and densely defined operator in a Hilbert space $\mathcal{H}$, $F(t,u)$ is a nonlinear operator in $\mathcal{H}$ for any fixed $t\ge 0$, and $f(t)$ is a function on $[0,\infty)$ with values in $\mathcal{H}$. 
In addition, it is assumed that 
\begin{equation}
\label{l2}
\Real\langle u, Au\rangle \le \gamma(t)\|u\|^2, \qquad \|F(t,u)\| \le \alpha(t)\|u\|^p,\quad p>1,\qquad \|f(t)\|\le \beta(t),
\end{equation}
for all $t\ge 0$ where $\langle \cdot,\cdot\rangle$ and $\|\cdot\|$ denote the inner product and the norm in $\mathcal{H}$, respectively. 
Using equation \eqref{l1} and inequalities in \eqref{l2}, one obtains the following inequality (cf. \cite{R1}, \cite{R4})
\begin{equation*}
\dot{g}(t) \le \gamma(t)g(t) + \alpha(t)g^p(t) + \beta(t),\qquad t\ge 0,\quad g(t) := \|u(t)\|,\quad p>1.
\end{equation*}
It was proved in \cite[Lemma 1]{R4} that if there exists a function $\mu(t)>0$ such that 
\begin{equation}
\label{l3}
\alpha(t)\frac{1}{\mu^p(t)} + \beta(t) \le \frac{1}{\mu(t)} \bigg(-\gamma(t) - \frac{\dot{\mu}(t)}{\mu(t)}\bigg),\qquad t\ge 0,\qquad \|u(0)\|\mu(0)\le 1,
\end{equation}
then
\begin{equation}
\label{l4}
0\le g(t) \le \frac{1}{\mu(t)},\qquad \forall t\ge 0.
\end{equation}
Using this result, it was proved in \cite{R4} that if $\int_0^\infty\big[\gamma(t) + \alpha(t)\big]\, dt$ is not `large', then the equilibrium solution $u=0$ to the equation $\dot{u} = F(t,u)$ is Lyapunov stable under persistently acting perturbation $f(t)$. Namely, given any arbitrarily small $\epsilon>0$, if $\|f(t)\|$ is sufficiently small, then there exists $\delta>0$ such that if $\|u(0)\|<\delta$, then $\|u(t)\| <\epsilon$ for all $t\ge 0$. Other results for the boundedness of $\|u(t)\|$ were also obtained by using inequalities \eqref{l3} and \eqref{l4}. 

Several stability results for equation \eqref{l1} were also obtained in \cite{NH}.
One of the results states that  if $\int_0^\infty\big[\gamma(t) + \alpha(t)\big]\, dt < \infty$, then the equilibrium solution $u=0$ is Lyapunov stable under persistently acting perturbation $f(t)$. This result is stronger than the one in \cite{R4} as the function $\gamma(t)$ can take positive and negative values and $\int_0^\infty\big[\gamma(t) + \alpha(t)\big]\, dt$ can be arbitrarily large as long as it is finite. 
Other sufficient conditions for the boundedness of $u(t)$ and for $\lim_{t\to}u(t)=0$ were also proposed and justified in \cite{NH}. 
The advantage of the results in \cite{NH} compared to those in \cite{R4} is that: one does not have to find a function $\mu(t)>0$ which solves inequality \eqref{l3}. Moreover, the results in \cite{NH} are applicable to the case when $\gamma(t)$ takes both positive and negative values, for example, $\gamma(t)=\sin t$. This case is not easy to handle using the results in \cite{R4}. 
However, the results in \cite{NH} are not applicable to equations in Banach spaces. The objective of this paper is to extend the stability results in \cite{NH} for equations in Banach spaces and also for the case when the function $\gamma(t)$ is not in $L^1[0,\infty)$ for equation \eqref{l1}.

In this paper we study the stability of the solution to equation \eqref{eq1} under non-classical assumptions. 
The new results in this paper include Theorem \ref{thm2.1} in which we proved that the solution $v=0$ to equation $\dot{v} = F(t,v)$ is Lyapunov stable under persistently acting perturbation $f(x)$ if $\sup_{t\ge 0}\|U(t)\|<\infty$ and $\sup_{t\ge 0}\int_0^t \|U(t,\xi)\|\alpha(\xi)\, d\xi < \infty $ where $U(t):=U(t,0)$, and $U(t,\xi)$ is the solution to the equation $\frac{d}{dt}{U}(t,\xi) = F'_u(t,0)U(t,\xi)$, $t\ge \xi$, $U(\xi,\xi)=I$. Here $I$ is the identity operator in $\mathcal{X}$. 
A sufficient condition for the solution to equation \eqref{eq1} to be bounded is proposed and justified in Theorem \ref{thm2.3}. 
A consequence of this result is: 
If $\int_0^\infty \|U^{-1}(\xi)\|\|U(\xi)\|^p \alpha(\xi)\, d\xi <\infty$ and $\lim_{t\to\infty}\|U(t)\| = 0$, then 
the solution $u(t)$ to equation \eqref{eq1} satisfies $\lim_{t\to\infty}u(t)=0$ provided that $\|f(t)\|$ is sufficiently small. An estimate for the rate of growth/decay of $\|u(t)\|$ when $t$ tends to infinity is also given in Theorem \ref{thm2.3}. 
Stability of solutions to equations with unbounded operators in Hilbert spaces is also studied in Section \ref{section3}. 

\section{Equations in Banach spaces}

In this section we will study the stability of the solution to equation \eqref{eq1}. 
Let 
\begin{equation}
\label{defnB}
B(t):=F'_u(t,0),\qquad t\ge 0,
\end{equation}
where $F'_u(t,0):=F'_u(t,u)|_{u=0}$ and $F'_u(t,u)$ is the Fr\'echet derivative of $F(t,u)$ with respect to $u$. 
Let $b(t) := \|B(t)\|$ and assume that $b(t)$ is in $L^1_{loc}[0,\infty)$. 
Equation \eqref{eq1} can be written as
\begin{equation}
\label{2e1}
\dot{u} = B(t)u + G(t,u) + f(t),\qquad t\ge 0,\qquad u(0) = u_0,
\end{equation}
where
$$
G(t,u):= F(t,u) - F'_u(t,0)u.
$$
This and inequality \eqref{ae5} imply
\begin{equation}
\label{dec01}
\|G(t,u)\| \le \alpha(t) \|u\|^p,\qquad t\ge 0,\quad u\in \mathcal{X}.
\end{equation}

Let us first show that condition \eqref{ae5} holds if $F(t,0)=0$, $t\ge 0$, and the Fr\'echet derivative $F'_u(t,u)$ is H\"{o}lder continuous of order $q=p-1>0$ with respect to $u$. 
We have
\begin{equation*}
F(t,u) - F(t,0) = \int_0^1 F'_u(t,\xi u)u\, d\xi. 
\end{equation*}
Thus, if $F(t,0)=0$, then
\begin{equation}
\label{q7}
\begin{split}
G(t,u) &=  F(t,u) - F'_u(t,0)u =  F(t,u) - F(t,0) - \int_0^1 F'_u(t,0)u\, d\xi\\
          &= \int_0^1 F'_u(t,\xi u)u\, d\xi - \int_0^1 F'_u(t,0)u\, d\xi = \int_0^1 \bigg[F'_u(t,\xi u) - F'_u(t,0)\bigg]u\, d\xi. 
\end{split}
\end{equation}
Assume that 
\begin{equation}
\label{q8}
\|F'_u(t,v) - F'_u(t,w)\| \le (q+1)\alpha(t) \|v - w\|^{q},\qquad q = p-1>0,\quad v,w\in B(0,R)\subset \mathcal{X},
\end{equation}
where $B(0,R)$ is a ball in $\mathcal{X}$ centered at the origin of a sufficiently large radius $R>0$. 
This inequality means that the Fr\'echet derivative $F'_u(t,u)$ is H\"{o}lder continuous of order $q$ with respect to $u$ in the ball $B(0,R)\subset \mathcal{X}$. 
From \eqref{q7} and \eqref{q8} we get
\begin{equation}
\label{z10}
\begin{split}
\|G(t,u)\|  &\le \int_0^1 \|F'_u(t,\xi u) - F'_u(t,0)\| \|u\|\, d\xi \le (q+1)\alpha(t) \int_0^1 \|\xi u \|^{q}\, d\xi \|u\|\\
&= (q+1)\alpha(t) \int_0^1 \xi^{q}\, d\xi \|u\|^{q+1} = \alpha(t)\|u\|^{p},\qquad p=q+1.
\end{split}
\end{equation}
Therefore, inequalities \eqref{ae5} and \eqref{dec01} hold if $F(t,0)=0$, $t\ge 0$, and 
inequality \eqref{q8} holds. Thus, inequality \eqref{ae5} is not a restrictive one. 

Let $U(t)$ be the solution to the equation
\begin{equation}
\label{2e2}
\frac{d}{d t}U(t) =  B(t)U(t),\qquad t\ge 0,\qquad U(0) = I 
\end{equation}
where $B(t)$ is defined in \eqref{defnB} and $I$ is the identity operator in $\mathcal{X}$. 
The solution $U(t)$ to equation \eqref{2e2} exists globally and is unique if the function $b(t)=\|B(t)\|$ is in $L^1_{loc}[0,\infty)$  (see, e.g., \cite{DK}). 
Moreover, the inverse operator $U^{-1}(t)$ exists, for any fixed $t\ge 0$, and $U^{-1}(t)$ solves the following equation (see, e.g., \cite{DK}):
\begin{equation}
\label{2e3}
\frac{d}{d t}U^{-1}(t) =  -U^{-1}(t)B(t),\qquad t \ge 0,\qquad U^{-1}(0) = I.
\end{equation}
Equations \eqref{2e1} and \eqref{2e3} imply
\begin{equation}
\label{eq13*}
\frac{d}{dt}\bigg(U^{-1}(t)u(t)\bigg) = U^{-1}(t)\bigg(G(t,u(t)) + f(t)\bigg),\qquad t\ge 0. 
\end{equation}
Integrate this equation from $0$ to $t$ to get
$$
U^{-1}(t)u(t) -  U^{-1}(0)u(0)  = \int_0^t U^{-1}(\xi)\bigg(G(\xi,u(\xi)) + f(\xi)\bigg)\, d\xi,\qquad t\ge 0. 
$$
This implies
\begin{equation}
\label{2e4}
u(t) = U(t)u_0 + U(t)\int_0^t U^{-1}(\xi)\big[G(\xi,u(\xi)) + f(\xi)\big]\, d\xi.
\end{equation}
Here, we have used the relations $U^{-1}(0) = I$ and $u(0) = u_0$. 
Equation \eqref{2e4}, the triangle inequality, inequality \eqref{dec01}, and the relation $\beta(t) = \|f(t)\|$ imply 
\begin{equation}
\label{zx1}
\|u(t)\| \le \|U(t)\|\|u_0\| + \int_0^t \|U(t)U^{-1}(\xi)\|\bigg(\alpha(\xi)\|u(\xi)\|^p + \beta(\xi)\bigg)\, d\xi.
\end{equation}

\begin{remark}{\rm 
From equation \eqref{2e2} one can verify that the operator $U(t,\xi):=U(t)U^{-1}(\xi)$ is the solution to the equation
\begin{equation}
\label{eqzx1}
\frac{\partial }{\partial t}U(t,\xi) = B(t)U(t,\xi),\qquad t\ge \xi,\qquad U(\xi,\xi) = I. 
\end{equation}
Also, we have
\begin{equation*}
\|U(t)U^{-1}(\xi)\| \le \|U(t)\| \|U^{-1}(\xi)\|,\qquad t,\xi\ge 0.
\end{equation*}
}
\end{remark}

Throughout this section, we assume that the following assumption holds:

\vskip 0.1in
\noindent
{\bf Assumption A.} The equation
$$
\dot{u} = F(t,u) + f(t),\quad t\ge t_0,\quad u(t_0)= \tilde{u}_0,\quad \dot{u} := \frac{d u}{dt}
$$
has a unique local solution for any $t_0\ge 0$ and $\tilde{u}_0\in B(0,R)\subset \mathcal{X}$ where $R>0$ is sufficiently large.

\vskip 0.2in
\noindent 
{\bf Assumption A1.} The following inequalities hold
\begin{equation}
\label{ieq27}
\sup_{t\ge 0} \|U(t)\| <\infty,\qquad M:=\sup_{t\ge 0}\int_0^t \|U(t,\xi)\|\alpha(\xi)\, d\xi <\infty,
\end{equation}
where $U(t):=U(t,0)$, $U(t,\xi)$ is the solution to equation \eqref{eqzx1}, and the function $\alpha(t)$ is from inequality \eqref{ae5}. 
\vskip 0.2in 

\begin{remark}{\rm
The first inequality in Assumption {\bf A1} is a necessary condition for $u(t)$ to be bounded. If this inequality does not hold, then the solution to equation \eqref{eq1} is unbounded even for the case when $\alpha(t)\equiv 0$, in general. 
The second inequality in Assumption {\bf A1} means that the contribution from the nonlinear term $G(t,u)$ is not too large as $t$ tends to infinity. 
}
\end{remark}

\begin{theorem}
\label{thm2.1}
Let Assumptions {\bf A} and {\bf A1} hold and let $u(t)$ be the solution to equation \eqref{eq1}. 
Given any arbitrarily small $\epsilon>0$, if $\|f(t)\|$ is sufficiently small, then there exists $\delta>0$ such that if $\|u(0)\|<\delta$, then $\|u(t)\| <\epsilon$ for all $t\ge 0$.
\end{theorem}


\begin{remark}{\rm
Equation \eqref{eq1} can be considered as a perturbed equation of the following equation
\begin{equation}
\label{unpertubed}
\dot{v} = F(t,v),\qquad t\ge 0. 
\end{equation}
Since $F(t,0) = 0$, $\forall t\ge0$, the function $v(t)\equiv 0$ is an equilibrium solution to equation \eqref{unpertubed}. 
The function $f(t)$ in \eqref{eq1} can be considered as persistently acting perturbation to equation \eqref{unpertubed}. 
Under these terminologies, the conclusion of Theorem \ref{thm2.1} can be rephrased as 
the equilibrium solution $v=0$ to equation \eqref{unpertubed} is Lyapunov stable under persistently acting perturbations $f(t)$. 
}
\end{remark}

\begin{proof}[Proof of Theorem \ref{thm2.1}]
It follows from the second inequality in \eqref{ieq27} that if 
$\epsilon>0$ is sufficiently small then one gets
\begin{equation}
\label{q2}
\epsilon^{p-1}\sup_{t\ge 0}\int_0^t \|U(t,\xi)\|\alpha(\xi)\, d\xi < \frac{1}{4},\qquad p>1.
\end{equation}
Since $\sup_{t\ge 0} \|U(t)\| <\infty$ by \eqref{ieq27}, one can choose $\delta>0$ sufficiently small so that
\begin{equation}
\label{eq5}
\delta \sup_{t\ge 0}\|U(t)\| < \frac{\epsilon}{4},\qquad \delta < \epsilon.
\end{equation}

{\it Let us prove that if $0\le \|u_0\|<\delta$ and $\beta(t) = \|f(t)\|$ is sufficiently small, then $\|u(t)\|<\epsilon$ for all $t\ge 0$.} 

Choose $f(t)$ so that
\begin{equation*}
\beta(t) := \|f(t)\| <  \frac{\epsilon \alpha(t)}{4M},\qquad t\ge 0.
\end{equation*}
This and the second inequality in \eqref{ieq27} imply
\begin{equation}
\begin{split}
\label{eq6}
\int_0^t \|U(t,\xi)\|\beta(\xi)\, d\xi &<  \frac{\epsilon}{4M}\int_0^t \|U(t,\xi)\|\alpha(\xi)\, d\xi \\
& \le \frac{\epsilon}{4M} \sup_{t\ge 0 }\int_0^t \|U(t,\xi)\|\alpha(\xi)\, d\xi \le 
\frac{\epsilon}{4M}M = \frac{\epsilon}{4},\qquad \forall t\ge 0.
\end{split}
\end{equation}

{\it Let $[0,\tilde{T})$ be the maximal interval of existence of the solution $u(t)$ to equation \eqref{eq1}.} 
It is clear that $\tilde{T}>0$ by Assumption A. 
Let us prove that $\tilde{T}=\infty$, i.e., the solution $u(t)$ exists globally. Assume the contrary that $\tilde{T}$ is finite. 
Let $T>0$ be the largest value such that
\begin{equation}
\label{eqa1}
\|u(t)\| \le \epsilon,\qquad \forall t\in [0,T).
\end{equation}
Since $[0,\tilde{T})$ is the maximal interval of existence of $u(t)$, one has $0<T\le \tilde{T}$.  

{\it Let us prove that $T\ge\tilde{T}$}. Assume the contrary that $T<\tilde{T}$. 
Thus, $T$ is finite and, by the continuity of $g(t)$ and Assumption {\bf A} we have
\begin{equation}
\label{eqs2}
\|u(T)\| = \epsilon.
\end{equation}
Indeed, if $\|u(T)\| < \epsilon$, then by using Assumption {\bf A} one can extend the solution $u(t)$ to a larger interval, say, $[0,T+\theta)$, for some $\theta>0$, and $\|u(t)\| \le \epsilon,\forall t\in [0,T+\theta)$. This contradicts the definition of $T$. 
From inequalities \eqref{zx1} and \eqref{eqa1} one gets
$$
\|u(t)\|\le \|U(t)\| \|u_0\| + \int_0^t \|U(t,\xi)\| [\alpha(\xi)\epsilon^{p} + \beta(\xi)]\, d\xi,\qquad 0\le t< T.
$$
This, the inequality $\|u_0\| < \delta$, and inequalities \eqref{q2}, \eqref{eq5}, and \eqref{eq6} imply
\begin{equation*}
\begin{split}
\|u(t)\|&\le \sup_{t\ge 0}\|U(t)\|\delta + \epsilon \int_0^t \|U(t,\xi)\| \alpha(\xi)\epsilon^{p-1}\, d\xi  + \int_0^t \|U(t,\xi)\| \beta(\xi)\, d\xi \\
&\le \frac{\epsilon}{4} + \frac{\epsilon}{4}+ \frac{\epsilon}{4} = \frac{3\epsilon}{4},\qquad 0\le t< T.
\end{split}
\end{equation*}
This and the continuity of $u(t)$ imply $\|u(T)\| \le \frac{3\epsilon}{4}$ which contradicts relation \eqref{eqs2}. 
This contradiction implies that $T\ge\tilde{T}$, i.e.,
\begin{equation}
\label{q3}
\|u(t)\| \le \epsilon,\qquad 0\le t <\tilde{T}.
\end{equation}

It follows from inequality \eqref{q3} and the continuity of $u(t)$ that $\|u(\tilde{T})\| \le \epsilon$. 
Using the inequality $\|u(\tilde{T})\| \le \epsilon$ and Assumption {\bf A}, one obtains the existence of $u(t)$ on a larger interval, say, $[0,\tilde{T} + \theta)$ for some $\theta >0$. This contradicts the definition of $[0,\tilde{T})$ as the maximal interval of existence of $u(t)$. The contradiction implies that $\tilde{T} = \infty$. In addition, one has $\|u(t)\| \le \epsilon$, $\forall t\ge 0$. 
Theorem \ref{thm2.1} is proved. 
\end{proof}

Now we consider the following problem: Given the nonlinear operator $F(t,u)$, i.e., given $F'_u(t,0)$ and $\alpha(t)$, under what conditions on $f(t)$ does the solution to problem \eqref{eq1} exist globally, is the solution bounded, and does it decay to zero as $t\to\infty$? An answer to this question is given in Assumption {\bf A2} and this answer is justified in Theorem \ref{thm2.3} below. 

\vskip 0.2in
\noindent 
{\bf Assumption A2.} The following inequality holds
\begin{equation}
\label{eqn1}
\omega :=\sup_{t\ge 0} \bigg(\frac{\beta(t)}{\alpha(t)}\bigg)^\frac{1}{p} \frac{1}{\|U(t)\|} 
< \frac{1}{\bigg[(p-1)\int_0^\infty \|U^{-1}(\xi)\|\|U(\xi)\|^p \alpha(\xi) \,d\xi \bigg]^{\frac{1}{p-1}}} - \|u(0)\|,
\end{equation}
where $U(t)$ and $U^{-1}(t)$ are defined by \eqref{2e2} and \eqref{2e3} 
while $\alpha(t)$ and $\beta(t)$ are defined from \eqref{ae5} and \eqref{ae6}.

\begin{remark}{\rm 
It follows from \eqref{eqn1} that
$$
\|u(0)\| + \omega < \frac{1}{\bigg[(p-1)\int_0^\infty \|U^{-1}(\xi)\|\|U(\xi)\|^p \alpha(\xi) \,d\xi \bigg]^{\frac{1}{p-1}}}.
$$
This implies
\begin{equation}
\label{eq11}
\frac{1}{\big(\|u(0)\| + \omega\big)^{p-1}} > (p-1)\int_0^\infty \|U^{-1}(\xi)\|\|U(\xi)\|^p \alpha(\xi) \,d\xi.
\end{equation}
Moreover, from the definition of $\omega$ in \eqref{eqn1} one gets
\begin{equation}
\label{eqs6}
\frac{\beta(t)}{\alpha(t) \|U(t)\|^p} \le \sup_{t\ge 0}\frac{\beta(t)}{\alpha(t) \|U(t)\|^p} = \omega^p,\qquad t\ge 0.
\end{equation}
}
\end{remark}

Let $V:[0,\infty)\to \mathcal{X}$ be a differentiable function of $t$ with values in $\mathcal{X}$. 
From the triangle inequality one gets
$$
\bigg|\|V(t+\delta)\| - \|V(t)\| \bigg|\le \|V(t+\delta) - V(t)\|,\qquad t\ge 0.
$$
This implies
\begin{equation}
\label{z32}
\frac{d}{dt} \|V(t)\| \le \bigg{\|} \frac{d}{dt} V(t)\bigg{\|},\qquad t\ge 0,
\end{equation}
where the derivative of $\|V(t)\|$ at its zeros is understood as the right derivative. 

\begin{theorem}
\label{thm2.3}
Let Assumptions {\bf A} and {\bf A2} hold. 
Then the solution $u(t)$ to problem \eqref{eq1} exists globally and satisfies
\begin{equation}
\label{eqa7}
\|u(t)\| \le C_2 \|U(t)\|,\qquad t\ge 0,\quad C_2 = const>0.
\end{equation}
Consequently, 
\begin{enumerate}
\item[i)]{}if
\begin{equation}
\label{eq10}
\sup_{t\ge 0} \|U(t)\| < \infty,
\end{equation}
then the solution $u(t)$ is bounded;
\item[ii)]{}if
\begin{equation}
\label{eq12.0}
\lim_{t\to\infty} \|U(t)\| = 0,
\end{equation}
then
\begin{equation}
\label{eq12}
\lim_{t\to\infty} u(t) = 0.
\end{equation}
\end{enumerate}
\end{theorem}


\begin{proof}
Let us prove that the solution $u(t)$ to equation \eqref{eq1} exists globally. Assume the contrary that the maximal interval of existence of $u(t)$ is a finite interval, namely, $[0,T)$. 
From inequality \eqref{z32} with $V(t)$ replaced by $U^{-1}(t)u(t)$, equation \eqref{eq13*}, inequality \eqref{z10}, 
and the identity $u(t) = U(t)U^{-1}(t)u(t)$, one gets
\begin{equation*}
\begin{split}
\frac{d}{dt}\|U^{-1}(t)u(t)\| &\le \bigg\|\frac{d}{dt}\bigg(U^{-1}(t)u(t)\bigg)\bigg\| = \bigg\| U^{-1}(t)\bigg(G(t,u) + f(t)\bigg)\bigg\|\\
& \le \| U^{-1}(t)\| \bigg(\|G(t,u)\| + \|f(t)\|\bigg)\\
&\le \|U^{-1}(t)\|\bigg(\alpha(t)\|u(t)\|^p + \beta(t)\bigg)\\
& = \|U^{-1}(t)\|\bigg(\alpha(t)\|U(t)U^{-1}(t)u(t)\|^p + \beta(t)\bigg).
\end{split}
\end{equation*}
This and the inequality $\|U(t)U^{-1}(t)u(t)\| \le \|U(t)\| \|U^{-1}(t)u(t)\|$ imply
\begin{equation}
\label{dec5}
\begin{split}
\frac{d}{dt}\|U^{-1}(t)u(t)\| 
&\le \|U^{-1}(t)\|\bigg(\alpha(t)\|U(t)\|^p\|U^{-1}(t)u(t)\|^p + \beta(t)\bigg)\\
&= \alpha(t) \|U^{-1}(t)\| \|U(t)\|^p\bigg(\|U^{-1}(t)u(t)\|^p + \frac{\beta(t)}{\alpha(t)\|U(t)\|^p}\bigg).
\end{split}
\end{equation}
Inequalities \eqref{dec5} and \eqref{eqs6} and the inequality $a^p + b^p \le (a+b)^p,\,\forall a,b\ge 0, p>1$ imply
\begin{equation}
\label{zx4}
\begin{split}
\frac{d}{dt}\|U^{-1}(t)u(t)\|
&\le \alpha(t)\|U^{-1}(t)\|\|U(t)\|^p\bigg(\|U^{-1}(t)u(t))\|^p + w^p\bigg)\\
&\le \alpha(t) \|U^{-1}(t)\|\|U(t)\|^p\bigg(\|U^{-1}(t)u(t)\| + w\bigg)^p,\qquad 0\le t< T.
\end{split}
\end{equation}
Inequality \eqref{zx4} can be rewritten as
$$
\frac{d}{dt}\bigg( \frac{\big(\|U^{-1}(t)u(t)\| + \omega\big)^{1-p} }{1-p} \bigg) \le \alpha(t)\|U^{-1}(t)\|\|U(t)\|^p,\qquad 0\le t< T.
$$
Integrate this inequality from 0 to $t$ and use $U(0) = I$ to get
\begin{equation*}
\frac{\big(\|U^{-1}(t)u(t)\| + \omega\big)^{1-p} - \big(\|u(0)\| + \omega\big)^{1-p}}{1-p}  
\le \int_0^t \alpha(\xi)\|U^{-1}(\xi)\|\|U(\xi)\|^p\, d\xi, 
\end{equation*}
for all $t\in [0,T)$. 
This implies
\begin{equation}
\label{eqx20}
\big(\|U^{-1}(t)u(t)\| + \omega\big)^{p-1}  \le \frac{1}{(\|u(0)\| + \omega)^{1-p} - (p-1) \int_0^t \alpha(\xi) \|U^{-1}(\xi)\| \|U(\xi)\|^p\, d\xi}, 
\end{equation}
for all $t\in [0,T)$. 
Inequality \eqref{eq11} implies that the right-hand side of \eqref{eqx20} is well-defined for all $t\ge 0$. 
Thus, from \eqref{eqx20} one gets
\begin{equation*}
\begin{split}
\big(\|U^{-1}(t)u(t)\| + \omega\big)^{p-1}  
&\le \sup_{t\ge 0} \frac{1}{(\|u(0)\| + \omega)^{1-p} - (p-1) \int_0^t \alpha(\xi) \|U^{-1}(\xi)\| \|U(\xi)\|^p\, d\xi}\\
&= \frac{1}{(\|u(0)\| + \omega)^{1-p} - (p-1) \int_0^\infty \alpha(\xi)\|U^{-1}(\xi)\|\|U(\xi)\|^p\, d\xi}:=M_3,
\end{split}
\end{equation*}
for all $t\in [0,T)$. 
This implies
\begin{equation}
\label{eq18}
\|U^{-1}(t)u(t)\| \le M_3^{\frac{1}{p-1}} - \omega,\qquad 0\le t<T.
\end{equation}
It follows from \eqref{eq18} that
\begin{equation}
\label{eq19}
\begin{split}
\|u(t)\| & = \|U(t)U^{-1}(t)u(t)\| \\
&\le \|U(t)\|\|U^{-1}(t)u(t)\| \le \|U(t)\|\bigg(M_3^{\frac{1}{p-1}} - \omega\bigg) ,\qquad 0\le t< T. 
\end{split}
\end{equation}
Inequality \eqref{eq19} and the continuity of $u(t)$ imply that $\|u(T)\|\le \|U(T)\|\big(M_3^{\frac{1}{p-1}} - \omega\big)$. This and Assumption {\bf A} imply the existence of $u(t)$ on a larger interval, say, $[0,T+\delta)$ for some $\delta>0$. This contradicts the definition of $[0,T)$ as the maximal interval of existence of $u(t)$. The contradiction implies that $T=\infty$. This and inequality \eqref{eq19} imply inequality \eqref{eqa7}. 

If inequality \eqref{eq10} holds, then it follows from inequality \eqref{eqa7} that the solution $u(t)$ is bounded. 

If equation \eqref{eq12.0} holds, then relation \eqref{eq12} follows directly from inequality \eqref{eqa7}. Theorem \ref{thm2.3} is proved. 
\end{proof}

\begin{remark}{\rm
As we have mentioned earlier, inequality \eqref{eq10} is necessary for the boundedness of $u(t)$. In addition, if the right-hand side of \eqref{eqn1} is not positive, then the solution $u(t)$ may blow up at a finite time. 
Let us verify this claim by considering the following first-order ordinary equation:
\begin{equation}
\label{l5}
\dot{u} = \gamma(t)u(t) + \alpha(t)u^p(t),\qquad t\ge 0,\qquad u(0) = u_0>0. 
\end{equation}
One can verify that the solution to the equation
$$
\dot{U}(t) = \gamma(t)U(t),\qquad U(0) = 1,
$$
is
$$
U(t) = e^{\int_0^t \gamma(\xi)\, d\xi}.
$$
In addition $U^{-1}(t) = e^{-\int_0^t \gamma(\xi)\, d\xi}$ and we have
\begin{equation}
\label{beq40}
\frac{d}{dt}U^{-1}(t) = -\gamma(t)U^{-1}(t),\qquad t\ge 0.
\end{equation}
It follows from the product rule, equation \eqref{beq40}, and equation \eqref{l5} that
\begin{equation*}
\begin{split}
\frac{d}{dt}\big(U^{-1}(t)u(t)\big) &= u(t)\frac{d}{dt}U^{-1}(t) + U^{-1}(t)\frac{d}{dt}u(t)\\
&= -\gamma(t)U^{-1}(t)u(t) + U^{-1}(t)[\gamma(t)u(t) + \alpha(t)u^p(t)]\\
&= U^{-1}(t)\alpha(t)u^p(t) = U^{p-1}(t)\alpha(t)\big[U^{-1}(t)u(t)\big]^p. 
\end{split}
\end{equation*}
This implies
$$
\frac{1}{\big[U^{-1}(t)u(t)\big]^p}\cdot\frac{d}{dt}\big(U^{-1}(t)u(t)\big) =  U^{p-1}(t)\alpha(t).
$$
Integrate this equation from 0 to $t$ to get
$$
\frac{\big(U^{-1}(t)u(t)\big)^{1-p} - \big(U^{-1}(0)u(0)\big)^{1-p}}{1-p}  =  \int_0^t U^{p-1}(\xi)\alpha(\xi)\, d\xi.
$$
From this equation one obtains
$$
u(t)=\tilde{u}(t) := U(t)\bigg(\frac{1}{u^{1-p}_0 - (p-1) \int_0^t U^{p-1}(\xi)\alpha(\xi)\, d\xi}\bigg)^{\frac{1}{p-1}}.
$$
The function $\tilde{u}(t)$ blows up at a finite time $t=t_0$ if $t_0$ is the solution to the equation 
$$
0=\frac{1}{u^{p-1}_0} - (p-1) \int_0^t U^{p-1}(\xi)\alpha(\xi)\, d\xi := \varphi(t).
$$
The function $\varphi$ is a continuous function on $[0,\infty)$, $\varphi(0) = 1/u_0^{p-1}>0$. Thus, by the Intermediate Value Theorem $\varphi(t) = 0$ has a solution $t=t_0$ if $\lim_{t\to\infty}\varphi(t)<0$, i.e., 
$$
\frac{1}{u^{p-1}_0} - (p-1) \int_0^\infty U^{p-1}(\xi)\alpha(\xi)\, d\xi <0.
$$
This inequality is equivalent to
$$
\frac{1}{\bigg[(p-1) \int_0^\infty U^{p-1}(\xi)\alpha(\xi)\, d\xi\bigg]^{\frac{1}{p-1}}} - u_0 < 0.
$$

Thus, the condition that the right-hand side of \eqref{eqn1} is positive is not a restrictive one. When the right-hand side of \eqref{eqn1} is positive, 
\eqref{eqn1} holds true if the fraction $\beta(t)/\alpha(t)$ is sufficiently small. 
In particular, it holds true if $\beta(t)\equiv 0$, i.e., $f$ is absent from equation \eqref{eq1}. 
}
\end{remark}


From Theorem \ref{thm2.3} we have the following corollary:

\begin{corollary}
\label{corrollary2.5}
Let Assumption {\bf A} hold. 
Assume that 
\begin{equation}
\label{nov3}
w:=\sup_{t\ge 0} \bigg(\frac{\beta(t)}{\alpha(t)}\bigg)^\frac{1}{p} \frac{1}{\|U(t)\|} < \bigg(\frac{1}{(p-1)\int_0^\infty \alpha(\xi)\|U^{-1}(\xi)\| \|U(\xi)\|^p\, d\xi}\bigg)^{\frac{1}{p-1}},\qquad p>1. 
\end{equation}
If $\|u_0\|$ is sufficiently small so that
\begin{equation}
\label{nov4}
\|u_0\| < \bigg(\frac{1}{(p-1)\int_0^\infty \alpha(\xi)\|U^{-1}(\xi)\| \|U(\xi)\|^p\, d\xi}\bigg)^{\frac{1}{p-1}} - \omega,
\end{equation}
then the solution $u(t)$ to problem \eqref{eq1} exists globally and satisfies the estimate
\begin{equation}
\label{q55}
\|u(t)\| \le C_2 \|U(t)\|,\qquad t\ge 0,\qquad C_2 = const>0.
\end{equation}
Consequently, 
\begin{enumerate}
\item[(i)]{}
if
\begin{equation}
\label{eq2020x1}
\sup_{t\ge 0} \|U(t)\| < \infty,
\end{equation}
then the solution $u(t)$ is bounded;
\item[(ii)]{}
if
\begin{equation*}
\label{}
\lim_{t\to\infty}\|U(t)\| = 0,
\end{equation*}
then 
\begin{equation*}
\lim_{t\to\infty} u(t) = 0.
\end{equation*}
\end{enumerate}
\end{corollary}

\begin{proof}
It follows from inequalities \eqref{nov3} and \eqref{nov4} that inequality \eqref{eqn1} in Assumption A2 holds true. Consequently, Corollary \ref{corrollary2.5} follows from Theorem \ref{thm2.3}. 
\end{proof}

\begin{remark}
{\rm 
Assume that
\begin{equation}
\label{fe09}
\int_0^\infty \alpha(\xi)\|U^{-1}(\xi)\| \|U(\xi)\|^p\, d\xi < \infty.
\end{equation}
Consider equation \eqref{eq1} with $f(t) \equiv 0$. Then $\beta(t)\equiv 0$ and inequality \eqref{nov3} holds true. Note that in this case, $u(t) \equiv 0$ is an equilibrium solution. If $\lim_{t\to\infty}\|U(t)\| = 0$ and $u(0)=u_0$ satisfies inequality \eqref{nov4} with $\omega =0$, then it follows from Corollary \ref{corrollary2.5} that $\lim_{t\to\infty}u(t) = 0$. This means that the equilibrium solution $u(t) = 0$ of the equation $\dot{u} = F(t,u)$ is asymptotically stable if inequality \eqref{fe09} holds and $\lim_{t\to\infty}\|U(t)\| = 0$. 
}
\end{remark}

\begin{corollary}
\label{cor2.9}
Let Assumption {\bf A} hold. Assume that 
\begin{equation}
\label{nov1}
\lim_{t\to\infty} \|U(t)\| = 0,\qquad \int_0^\infty \alpha(t)\, dt <\infty \qquad\|U^{-1}(t)\| \le \frac{c}{\|U(t)\|},\qquad \forall t\ge 0,
\end{equation}
where $c>0$. 
Then the solution $v(t)\equiv0$ to the equation $\dot{v} = F(t,v)$, $t\ge 0$, $v(0) = 0$ is asymptotically stable under persistently acting perturbations $f$. 
Namely, if $\|f(t)\|$ is sufficiently small, then there exists $\delta>0$ such that if $\|u(0)\|<\delta$, then $\lim_{t\to\infty}\|u(t)\|=0$ where $u(t)$ is the solution to equation \eqref{eq1}.
\end{corollary}

\begin{proof}
It follows from the third inequality in \eqref{nov1} that 
\begin{equation}
\label{nov2}
\int_0^\infty \alpha(t) \|U^{-1}(t)\|\|U(t)\|^p\, dt \le 
c\int_0^\infty \alpha(t) \|U(t)\|^{p-1}\, dt.
\end{equation}
Since $\|U(t)\|$ is a continuous function of $t$, it follows from the relation $\lim_{t\to\infty} \|U(t)\| = 0$ in \eqref{nov1}  that $\|U(t)\|$ is bounded on $[0,\infty)$. This and the second inequality in \eqref{nov1} imply that the integral in the right-hand side of \eqref{nov2} is finite. Thus, from inequality \eqref{nov2} one gets
\begin{equation*}
\int_0^\infty \alpha(\xi) \|U^{-1}(\xi)\|\|U(\xi)\|^p\, d\xi  < \infty. 
\end{equation*}
Therefore, the right-hand side of inequality \eqref{nov3} is a finite positive constant. 
Let $\beta(t)=\|f(t)\|$ be sufficiently small so that 
$$
0 \le \beta(t) < \alpha(t)\|U(t)\|^p \bigg(\frac{1}{2}\bigg)^p\bigg(\frac{1}{(p-1)\int_0^\infty \alpha(\xi)\|U^{-1}(\xi)\| \|U(\xi)\|^p\, d\xi}\bigg)^{\frac{p}{p-1}},\qquad t\ge 0.
$$
This implies 
$$
\bigg(\frac{\beta(t)}{\alpha(t)}\bigg)^{\frac{1}{p}}\frac{1}{\|U(t)\|} < \frac{1}{2}\bigg(\frac{1}{(p-1)\int_0^\infty \alpha(\xi)\|U^{-1}(\xi)\| \|U(\xi)\|^p\, d\xi}\bigg)^{\frac{1}{p-1}},\qquad t\ge 0.
$$
Inequality \eqref{nov3} follows from this inequality. 
If $\|u_0\|$ is sufficiently small, then inequality \eqref{nov4} follows from inequality \eqref{nov3}. 
Therefore, by Corollary \ref{corrollary2.5}, inequality \eqref{q55} holds. From inequality \eqref{q55} and the first relation in \eqref{nov1} one gets 
$\lim_{t\to\infty}\|u(t)\| = 0$. This completes the proof of Corollary \ref{cor2.9}. 
\end{proof}

Now we are interested in the following question: Given the perturbation $f(t)$, under what conditions on the nonlinear part of $F(t,u)$, in general, and on the function $\alpha(t)$ in inequality \eqref{ae5}, in particular, does the solution to problem \eqref{eq1} exist globally, is the solution  bounded, and does the solution decay to zero as $t\to\infty$? An answer to this question is given in Assumption {\bf A3} and is justified in Theorem \ref{thm2.4} below.  

\vskip 0.2in
\noindent
{\bf Assumption A3.}
Let $\alpha(t)\ge 0$ satisfy the inequality
\begin{equation}
\label{eqs9}
\alpha(t) \le \frac{\kappa\beta(t)}{[(\kappa+1)\zeta(t)]^p},\qquad t\ge 0,\quad \kappa>0,\quad p>1,
\end{equation}
where
\begin{equation}
\label{eq22}
\zeta(t) := \|U(t)\|\|u(0)\| + \int_0^t \|U(t)U^{-1}(\xi)\|\beta(\xi)\, d\xi.  
\end{equation}
\vskip 0.2in

\begin{theorem}
\label{thm2.4}
Assume that assumptions {\bf A} and {\bf A3} hold and that $u(0)\not = 0$. 
Then the solution $u(t)$ to problem \eqref{eq1} exists globally and 
\begin{equation}
\label{eq23bn}
\|u(t)\| < (\kappa + 1)\zeta(t),\qquad  \forall t\ge 0,\quad \kappa>0.
\end{equation}
Consequently,
\begin{enumerate}
\item[(i)]{
if the function $\zeta(t)$ is bounded on $[0,\infty)$, 
then the solution $u(t)$ to problem \eqref{eq1} is bounded;}
\item[(ii)]{
if $\lim_{t\to\infty}\zeta(t) = 0$, 
then 
\begin{equation*}
\lim_{t\to\infty} u(t) = 0.
\end{equation*}}
\end{enumerate}
\end{theorem}

\begin{proof}
Recall from our earlier assumptions that the functions $\alpha(t)$, $\beta(t)$, and $\|B(t)\|$ (cf. \eqref{defnB}) are in $L^1_{loc}([0,\infty))$. Thus, the integrals 
$\int_0^t \|U(\xi)\|\, d\xi$ and $\int_0^t \|U(t)U^{-1}(\xi)\| \beta(\xi)\, d\xi$ are well-defined for all $t\ge 0$. Therefore, the function $\zeta(t)$ in \eqref{eq22} is well-defined on $[0,\infty)$. 

{\it Let us prove that the solution $u(t)$ to problem \eqref{eq1} exists globally.} 
Assume the contrary that the maximal interval of existence of $u(t)$ is a finite interval $[0,T)$, $0<T<\infty$. 
Let us first show that
\begin{equation}
\label{eq23}
\|u(t)\| < (\kappa+1)\zeta(t),\qquad 0\le t< T.
\end{equation}
Since $U(0) = I$, it follows from \eqref{eq22} with $t=0$ that $\|u(0)\|= \zeta(0)<(\kappa+1)\zeta(0)$. This and the continuity of $\|u(t)\|$ and $\zeta(t)$ imply the existence of $\theta>0$ such that $\|u(t)\|< (\kappa+1)\zeta(t)$, $\forall t\in [0, \theta)$. 
Let $T_1\in (0,T]$ be the largest value such that
\begin{equation}
\label{eq24}
\|u(t)\| < (\kappa+1)\zeta(t),\qquad \forall t \in [0,T_1).
\end{equation}

Let us prove that $T_1= T$. Assume the contrary that $T_1<T$. From the continuity of $\|u(t)\|$ and the definition of $T_1$, one gets
\begin{equation}
\label{eqa9}
\|u(T_1)\| = (\kappa+1)\zeta(T_1),\qquad \|u(t)\|<(\kappa + 1)\zeta(t),\qquad  0\le t<T_1. 
\end{equation}
Inequalities \eqref{zx1}, \eqref{eq24}, and \eqref{eqs9}  imply
\begin{equation}
\label{eq28}
\begin{split}
\|u(t)\| &\le \|U(t)\|\|u(0)\| + \int_0^t \|U(t)U^{-1}(\xi)\|\bigg(\alpha(\xi)\big[(\kappa + 1)\zeta(\xi)\big]^p + \beta(\xi)\bigg)\, d\xi \\
&\le \|U(t)\| \|u(0)\| + \int_0^t \|U(t)U^{-1}(\xi)\|\bigg(\kappa\beta(\xi) + \beta(\xi)\bigg)\, d\xi \\
&= \|U(t)\| \|u(0)\| + (\kappa+1)\int_0^t \|U(t)U^{-1}(\xi)\| \beta(\xi) \, d\xi 
,\qquad \forall t\in [0,T_1).
\end{split}
\end{equation}
It follows from inequality \eqref{eq28} and the continuity of $u(t)$ that
\begin{align*}
\|u(T_1)\| &\le  \|U(T_1)\| \|u(0)\| + (\kappa+1)\int_0^{T_1} \|U(T_1)U^{-1}(\xi)\| \beta(\xi) \, d\xi \\
 &<  (\kappa+1)\|U(T_1)\| \|u(0)\| + (\kappa+1)\int_0^{T_1} \|U(T_1)U^{-1}(\xi)\| \beta(\xi) \, d\xi = (\kappa+1)\zeta(T_1).
\end{align*}
This inequality contradicts the first equality in \eqref{eqa9}. The contradiction implies that $T_1=T$, i.e., inequality \eqref{eq23} holds. 

Inequality \eqref{eq23} and the continuity of $u(t)$ imply that $\|u(T)\|$ is finite.
Thus, by using Assumption A with $t_0=T$ one obtains the existence of $u(t)$ on a larger interval, say, $[0,T+\delta)$ for some sufficiently small $\delta>0$. This contradicts the definition of $[0,T)$ as the maximal interval of existence of $u(t)$. The contradiction implies that $T=\infty$, i.e., the solution $u(t)$ to equation \eqref{eq1} exists globally. 

Inequality \eqref{eq23bn} follows from inequality \eqref{eq23} when $T=\infty$. 
It follows from inequality \eqref{eq23bn} that if $\zeta(t)$ is bounded on $[0,\infty)$, then the solution $u(t)$ to equation \eqref{eq1} is bounded and that if $\lim_{t\to\infty}\zeta(t) = 0$, then $\lim_{t\to\infty}u(t) = 0$. 
 Theorem \ref{thm2.4} is proved.  
\end{proof}

\begin{corollary}
\label{thm2.7}
Assume that assumptions {\bf A} and {\bf A3} hold and that $u(0)\not = 0$.
If
\begin{equation}
\label{eqs30s}
\sup_{t\ge 0}\|U(t)\| < \infty,  \qquad  \sup_{t \ge 0}\int_0^t \|U(t)U^{-1}(\xi)\|\beta(\xi)\, d\xi <\infty,
\end{equation}
then the solution $u(t)$ to problem \eqref{eq1} exists globally and is bounded. 
In addition, if
\begin{equation}
\label{eqs4s}
\lim_{t\to\infty}\|U(t)\| = 0,\qquad \lim_{t\to\infty}\int_0^t \|U(t)U^{-1}(\xi)\|\beta(\xi)\, d\xi = 0,
\end{equation}
then 
\begin{equation*}
\lim_{t\to\infty} u(t) = 0.
\end{equation*}
\end{corollary}

\begin{proof}

If \eqref{eqs30s} holds, then it follows from \eqref{eq22} that $\zeta(t)$ is bounded on $[0,\infty)$. 
Similarly, if \eqref{eqs4s} holds, then $\lim_{t\to\infty}\zeta(t) =0$. 
Consequently, the conclusions of Corollary \ref{thm2.7} follow from Theorem \ref{thm2.4}. 
\end{proof}

\section{Equations in Hilbert spaces}
\label{section3}

Consider the equation
\begin{equation}
\label{jq1}
\dot{u} = A(t)u + F(t,u) + f(t),\qquad t\ge 0,\quad u(0) = u_0.
\end{equation}
Here, $F(t,u)$ is an operator function from $[0,\infty)\times \mathcal{H}$ to $\mathcal{H}$ where $\mathcal{H}$ is a Hilbert space, $f$ is a function on $[0,\infty)$ with values in $\mathcal{H}$, and $A(t)$ is a densely defined operator in $\mathcal{H}$ for any fixed $t\ge 0$. 
The operator $A(t)$ is not necessarily bounded in $\mathcal{H}$. 
Since the operator $A(t)$ is not necessarily bounded in $\mathcal{H}$, the stability results in the previous section are not applicable to equation \eqref{jq1}.
Assume that $F(t,0) = 0$. Thus, $u(t)\equiv 0$ is an equilibrium solution to the equation
\begin{equation*}
\dot{u} = A(t)u + F(t,u),\qquad t\ge 0. 
\end{equation*}

Let
\begin{equation*}
B(t):=F_u'(t,0):= F_u'(t,u)\big{|}_{u=0}
\end{equation*} 
where $F_u'(t,u)$ denotes the Fr\'echet derivative of $F(t,u)$ with respect to $u$. 
Then equation \eqref{jq1} can be written as
\begin{equation}
\label{jn1}
\dot{u} = B(t)u + A(t)u + G(t,u) + f(t),\qquad t\ge 0,\quad u(0) = u_0,
\end{equation}
where
$$
G(t,u):= F(t,u) - F'_u(t,0)u.
$$
Assume that
\begin{equation}
\label{dec3}
\|G(t,u)\| \le \alpha(t)\|u\|^{p},\qquad t\ge 0,\quad u\in \mathcal{H},\qquad p=q+1,\quad q>0,
\end{equation}
and
\begin{equation}
\label{dec4}
\|f(t)\| \le \beta(t),\qquad t\ge 0. 
\end{equation}
Here, $\|\cdot\|$ denotes the norm in $\mathcal{H}$. 

Again, since 
\begin{equation*}
F(t,u) = F(t,u) - F(t,0)=\int_0^1 F'_u(t,\xi u)u\, d\xi,
\end{equation*}
we have
\begin{equation}
\label{jn3}
G(t,u) =F(t,u) - F'_u(t,0)u= \int_0^1 \bigg[F'_u(t,\xi u) - F'_u(t,0)\bigg]u\, d\xi. 
\end{equation}
Thus, if $F'_u(t,u)$ is H\"{o}lder continuous of order $q>0$ with respect to $u$, i.e., 
$$
\|F'_u(t,v) - F'_u(t,w)\| \le (q+1)\alpha(t) \|v - w\|^{q},
$$
then from \eqref{jn3} we have
\begin{equation*}
\begin{split}
\|G(t,u)\|  &\le \int_0^1 \|F'_u(t,\xi u) - F'_u(t,0)\| \|u\|\, d\xi \le (q+1)\alpha(t) \int_0^1 \|\xi u\|^{q}\, d\xi\, \|u\|,\\
&\le (q+1)\alpha(t) \int_0^1 \xi^{q}\, d\xi\, \|u\|^{q+1} = \alpha(t)\|u\|^{p},\qquad p:=q+1.
\end{split}
\end{equation*}

Let $U(t)$ be the solution to the equation:
\begin{equation}
\label{q9}
\frac{d}{dt}U(t) = B(t)U(t),\qquad U(0) = I,
\end{equation}
where $I$ is the identity operator in $\mathcal{H}$.
This operator function $U(t)$ exists globally if $\|B(t)\|$ is in $L^1_{loc}[0,\infty)$ (see, e.g., \cite{DK}). The inverse operator $U^{-1}(t)$ exists for any fixed $t\ge 0$ and $U^{-1}(t)$ solves the equation (cf. \cite{DK})
\begin{equation}
\label{jn5}
\frac{d}{dt} U^{-1}(t) = - U^{-1}(t)B(t),\qquad U^{-1}(0) = I.
\end{equation}
Assume that 
\begin{equation}
\label{jn6}
\Real\langle [U^{-1}]^*U^{-1}Au,u\rangle \le 0,\qquad \forall u\in \Dom(A(t))\subset \mathcal{H},\quad t\ge 0.
\end{equation}
Here, $\langle \cdot ,\cdot\rangle$ denotes the inner product in $\mathcal{H}$ and $[U^{-1}]^*$ denotes the adjoint operator of $U^{-1}$.

Throughout this section, we assume that the following assumption holds. 
\vskip 0.15in
\noindent
{\bf Assumption B.} 
The equation 
\begin{equation*}
\dot{u} = A(t)u + F(t,u) + f(t),\qquad t\ge \tau \ge 0,\quad u(\tau) = u_\tau,
\end{equation*}
has a unique local solution for all $u_\tau \in \Dom(A(\tau))\subset \mathcal{H}$. 
\vskip 0.2in

From the product rule and equations \eqref{jn5} and \eqref{jn1} one gets
\begin{equation}
\label{j8}
\begin{split}
\frac{d}{dt}\bigg[U^{-1}(t)u(t)\bigg] & = \bigg[\frac{d}{dt}U^{-1}(t)\bigg] u(t) + U^{-1}(t)\frac{d}{dt} u(t) \\
&= -U^{-1}(t)B(t)u(t) + U^{-1}(t) \big[ B(t)u(t) + A(t)u + G(t,u) + f(t)\big]\\\
&= U^{-1}(t)A(t)u + U^{-1}(t)\big[G(t,u) + f(t)\big],\qquad t\ge 0.
\end{split}
\end{equation}
Take inner product of equation \eqref{j8} by $U^{-1}(t)u(t)$ to get
\begin{equation}
\label{j9}
\begin{split}
\big\langle  \frac{d}{dt} (U^{-1}u),U^{-1}u\big\rangle 
= \langle U^{-1}Au, U^{-1}u\rangle + \langle U^{-1}[G(t,u) + f(t)], U^{-1}u \rangle,\qquad t\ge 0.
\end{split}
\end{equation}
Differentiating both sides of the equation $\|U^{-1}u\|^2 = \langle U^{-1}u,U^{-1}u\rangle$ with respect to $t$ to get
\begin{equation}
\label{eql88}
\begin{split}
2\|U^{-1}u\| \frac{d}{dt}\|U^{-1}u\| &= \langle \frac{d}{dt}(U^{-1}u),U^{-1}u\rangle + \langle U^{-1}u,\frac{d}{dt}(U^{-1}u)\rangle\\
&= \langle \frac{d}{dt}(U^{-1}u),U^{-1}u\rangle + \overline{\langle \frac{d}{dt}(U^{-1}u),U^{-1}u\rangle} = 2\Real \langle \frac{d}{dt}(U^{-1}u),U^{-1}u\rangle.
\end{split}
\end{equation}
The property $\langle u, v\rangle = \overline{\langle v, u\rangle},\forall u,v\in \mathcal{H}$ was used in equation \eqref{eql88}.

Taking the real parts of both sides of \eqref{j9} and using \eqref{jn6}, one obtains
\begin{equation}
\label{q10}
\begin{split}
\Real \langle \frac{d}{dt}(U^{-1}u),U^{-1}u\rangle
 &= \Real \langle (U^{-1})^*U^{-1}Au, u\rangle + \Real \langle U^{-1}[G(t,u) + f(t)], U^{-1}u \rangle\\
&\le \Real \langle U^{-1}[G(t,u) + f(t)], U^{-1}u \rangle\\
&\le \|U^{-1}[G(t,u) + f(t)]\|\| U^{-1}u \|,\qquad t\ge 0.
\end{split}
\end{equation}
Here, we have used the inequality $\Real \langle u,v\rangle \le |\langle u,v\rangle| \le \|u\|\|v\|$. 
From equation \eqref{eql88} and inequality \eqref{q10} one gets
$$
\|U^{-1}u\|\frac{d}{dt}\|U^{-1}u\| \le \|U^{-1}[G(t,u) + f(t)]\|\| U^{-1}u \|,\qquad t\ge 0. 
$$
This implies
\begin{align}
\label{neqc90}
\frac{d}{dt}\|U^{-1}u\| \le \| U^{-1}[G(t,u) + f(t)]\| \le \| U^{-1}(t)\|\bigg(\|G(t,u)\| + \|f(t)\|\bigg).
\end{align}
Inequalities \eqref{neqc90}, \eqref{dec3} and \eqref{dec4} imply
\begin{equation}
\label{j10}
\begin{split}
\frac{d}{dt}\|U^{-1}(t)u(t)\| 
&\le \| U^{-1}(t)\|\bigg( \alpha(t)\|u(t)\|^p + \beta(t)\bigg),\qquad t\ge 0,\quad u(0) = u_0.
\end{split}
\end{equation}
Integrate equation \eqref{j10} from 0 to $t$ to get
\begin{equation}
\label{q12}
\|U^{-1}(t)u(t)\| \le \|u_0\| + \int_0^t \| U^{-1}(\xi)\|\bigg( \alpha(\xi)\|u(\xi)\|^p + \beta(\xi)\bigg)\, d\xi,\qquad t\ge 0.
\end{equation}
One has
\begin{equation*}
\begin{split}
\|u(t)\| &= \|U(t)U^{-1}(t)u(t)\| 
\le \|U(t)\| \|U^{-1}(t)u(t)\|,\qquad t\ge 0.
\end{split}
\end{equation*}
This and inequality \eqref{q12} imply
\begin{equation}
\label{q11}
\begin{split}
\|u(t)\| &\le \|U(t)\|\|u_0\| + \|U(t)\|\int_0^t \| U^{-1}(\xi)\|\bigg( \alpha(\xi)\|u(\xi)\|^p + \beta(\xi)\bigg)\, d\xi,\qquad t\ge 0.
\end{split}
\end{equation}

Inequality \eqref{q11} is similar to inequality \eqref{zx1} that was used in the proof of Theorem \ref{thm2.1}. 
From inequality \eqref{q11}, Assumption {\bf B}, and similar arguments as in Theorem \ref{thm2.1} one obtains the following result:

\begin{theorem}
\label{thm3.1}
Assume that Assumption {\bf B} holds and that 
\begin{equation*}
\sup_{t\ge 0} \|U(t)\| <\infty,\qquad \sup_{t\ge 0}\|U(t)\|\int_0^t \|U^{-1}(\xi)\|\alpha(\xi)\, d\xi <\infty.
\end{equation*}
Then the solution $u(t)\equiv 0$ is Lyapunov stable under persistently acting perturbation $f(t)$. 
Namely, given any arbitrarily small $\epsilon>0$, if $\|f(t)\|$ is sufficiently small, then there exists $\delta>0$ such that if $\|u(0)\|<\delta$, then $\|u(t)\| <\epsilon$ for all $t\ge 0$. 
\end{theorem}

\begin{remark}{\rm Since $u(t) = U(t)U^{-1}(t)u(t)$, 
from inequality \eqref{j10} one gets
\begin{equation}
\label{jn17}
\begin{split}
\frac{d}{dt} \|U^{-1}(t)u(t)\| 
&\le \| U^{-1}(t)\|\bigg( \alpha(t)\|U(t)U^{-1}(t)u(t)\|^p + \beta(t)\bigg)\\
&\le \| U^{-1}(t)\|\bigg(\alpha(t) \|U(t)\|^p\|U^{-1}(t)u(t)\|^p + \beta(t)\bigg) \\
&= \alpha(t)  \| U^{-1}(t)\| \|U(t)\|^p\bigg(\|U^{-1}(t)u(t)\|^p + \frac{\beta(t)}{\alpha(t) \|U(t)\|^p}\bigg) ,\qquad t\ge 0.
\end{split}
\end{equation}
This inequality is similar to inequality \eqref{dec5} used in the proof of Theorem \ref{thm2.3}.
}
\end{remark}

Again, we consider the following problem: Given the nonlinear operator $F(t,u)$, i.e., given $B(t)=F'_u(t,0)$ and $\alpha(t)$, under what conditions on $f(t)$ does the solution to problem \eqref{jq1} exist globally, is the solution bounded, and does the solution decay to zero as $t\to\infty$? An answer to this question is given in Theorem \ref{thm3.3} below. 

From inequality \eqref{jn17}, Assumption {\bf B}, and similar arguments as in 
Theorem \ref{thm2.3}, one can prove the following result:

\begin{theorem}
\label{thm3.3}
Assume that Assumption {\bf B} hold and that
\begin{equation}
\label{no90}
\sup_{t\ge 0} \bigg(\frac{\beta(t)}{\alpha(t)}\bigg)^\frac{1}{p} \frac{1}{\|U(t)\|} 
< \frac{1}{\bigg[(p-1)\int_0^\infty \|U^{-1}(\xi)\|\|U(\xi)\|^p \alpha(\xi) \,d\xi \bigg]^{\frac{1}{p-1}}} - \|u(0)\|.
\end{equation}
Then the solution $u(t)$ to problem \eqref{jq1} exists globally and satisfies the estimate:
\begin{equation*}
\|u(t)\| \le C_2 \|U(t)\|,\qquad t\ge 0,\quad C_2 = const>0.
\end{equation*}
Consequently, 
\begin{enumerate}
\item[(i)]{}
if
\begin{equation}
\label{no91}
\sup_{t\ge 0} \|U(t)\| < \infty,
\end{equation}
then the solution $u(t)$ is bounded;
\item[(ii)]{}
if
\begin{equation*}
\lim_{t\to\infty} \|U(t)\| = 0,
\end{equation*}
then
\begin{equation*}
\lim_{t\to\infty} u(t) = 0.
\end{equation*}
\end{enumerate}
\end{theorem}

Now we consider the following question: Given the perturbation $f(t)$ and the operator $B(t)=F_u'(t,0)$, under what conditions on the function $\alpha(t)$ in inequality \eqref{ae5} does the solution to problem \eqref{jq1} exist globally, is the solution  bounded, and does the solution decay to zero as $t\to\infty$? An answer to this question is given in Theorem \ref{thm3.4} below.

From inequality \eqref{q11}, Assumption {\bf B}, and similar arguments as in Theorem \ref{thm2.4} one obtains the following result:

\begin{theorem}
\label{thm3.4}
Assume that Assumption {\bf B} hold and that
\begin{equation*}
\alpha(t) \le \frac{\kappa\beta(t)}{\big[(\kappa + 1)\zeta(t)\big]^p},\qquad t\ge 0,\quad \kappa>0,\quad p>1,
\end{equation*}
where
\begin{equation*}
\zeta(t) := \|U(t)\|\|u(0)\| + \|U(t)\|\int_0^t \|U^{-1}(\xi)\|\beta(\xi)\, d\xi.  
\end{equation*}
Assume that $\|u(0)\|\not = 0$. 
Then the solution $u(t)$ to problem \eqref{jq1} exists globally and satisfies the estimate
\begin{equation*}
\|u(t)\| < (\kappa+1)\zeta(t),\qquad  \forall t\ge 0.
\end{equation*}
Consequently,
\begin{enumerate}
\item[(i)]{
if the function $\zeta(t)$ is bounded on $[0,\infty)$, 
then the solution $u(t)$ to problem \eqref{jq1} is bounded;}
\item[(ii)]{
if $\lim_{t\to\infty}\zeta(t) = 0$, 
then 
\begin{equation*}
\lim_{t\to\infty} u(t) = 0.
\end{equation*}}
\end{enumerate}
\end{theorem}

\section{Examples}

Let us apply the new stability results in this paper to some simple examples. 

{\bf Example 1.} Let $\mathcal{X} = \mathbb{R}^2$. Consider the equation
\begin{equation}
\label{eq88}
\dot{u} = T\begin{bmatrix}
\cos t &0 \\
0 & -\cos t
\end{bmatrix}
T^{-1} u + \alpha(t)u^p + f(t),
\end{equation}
where $T$ is a nonsingular 2-by-2 constant matrix, $f:[0,\infty)\to \mathcal{X}$ and $\alpha(t)$ is a non-negative and continuous function on $[0,\infty)$.

Equation \eqref{eq88} is  a particular case of equation \eqref{2e1} with 
$$
B(t) = T\begin{bmatrix} \cos t & 0\\ 0 & -\cos t\end{bmatrix}T^{-1},\qquad G(t,u) = \alpha(t)u^p. 
$$
For this function $G(t,u)$ one has $\|G(t,u)\| = \alpha(t)\|u\|^p$. 
Equation \eqref{eqzx1} for this operator $B(t)$ becomes
\begin{equation*}
\frac{\partial}{\partial t}U(t,\xi) =  T\begin{bmatrix}
\cos t &0 \\
0 & -\cos t
\end{bmatrix}
T^{-1}U(t,\xi),\qquad t\ge \xi,\qquad U(\xi,\xi) = \begin{bmatrix} 1 & 0\\ 0 & 1\end{bmatrix}.
\end{equation*}
The solution to this equation is
\begin{equation}
\label{eq109}
U(t,\xi) = T\begin{bmatrix}
e^{\sin t - \sin \xi} &0 \\
0 & e^{-\sin t + \sin \xi}
\end{bmatrix}
T^{-1}.
\end{equation}
This implies
\begin{equation}
\label{nov5}
\|U(t,\xi)\| \le \max_{t\ge\xi\ge 0}\big\{e^{\sin t - \sin \xi} ,e^{-\sin t + \sin \xi}\big\}\le e^2,\qquad \forall t\ge \xi\ge 0.
\end{equation}
Recall that $U(t,\xi) = U(t)U^{-1}(\xi)$ and $U(t) = U(t,0)$. This and equation \eqref{eq109} imply
\begin{equation}
\label{eqf10}
U(t) = T\begin{bmatrix}
e^{\sin t} &0 \\
0 & e^{-\sin t}
\end{bmatrix}
T^{-1}.
\end{equation}
From equation \eqref{eqf10} and inequality \eqref{nov5} we get
\begin{equation}
\label{ewq1}
\|U(t)\| = \max_{t\ge 0}\big\{e^{\sin t},e^{-\sin t}\big\} \le e,\qquad \sup_{t\ge 0}\int_0^t \|U(t,\xi)\| \alpha(\xi)\, d\xi \le e^2 \int_0^\infty \alpha(\xi)\, d\xi.
\end{equation}
Therefore, it follows from \eqref{ewq1} and Theorem \ref{thm2.1} that if $\int_0^\infty \alpha(t)\, dt<\infty$, then the solution $u=0$ is Lyapunov stable under persistently acting perturbation $f(t)$. 

If $\int_0^\infty \beta(t)\, dt < \infty$ and inequality \eqref{eqs9} holds, i.e.,
\begin{equation*}
\alpha(t) \le \frac{\kappa\beta(t)}{[(\kappa + 1)\zeta(t)]^p},\qquad t\ge 0,\quad \kappa >0,
\end{equation*}
where
$$
\zeta(t) := \|U(t)\|\|u(0)\| + \int_0^t \|U(t)U^{-1}(\xi)\|\beta(\xi)\, d\xi,
$$
then it follows from \eqref{eq22}, inequality \eqref{ewq1}, inequality \eqref{nov5}, and Theorem \ref{thm2.4} that
\begin{equation*}
\begin{split}
\|u(t)\| \le (\kappa + 1) \zeta(t) &= (\kappa + 1)\bigg[\|U(t)\|\|u(0)\| + \int_0^t \|U(t)U^{-1}(\xi)\|\beta(\xi)\, d\xi\bigg]\\
&\le (\kappa + 1) \bigg[e \|u_0\| + e^2 \int_0^\infty \beta(\xi)\, d\xi\bigg] < \infty,\qquad \forall t\ge 0.
\end{split}
\end{equation*}
This means that the solution to equation \eqref{eq88} is bounded. 

Note that if we define
\begin{equation*}
A := \begin{bmatrix}
\cos t &0 \\
0 & -\cos t
\end{bmatrix}
\end{equation*}
then 
\begin{equation*}
\langle A u, u\rangle \le \gamma(t)\|u\|^2,\qquad \forall u\in \mathbb{R}^2,
\end{equation*}
where $\gamma(t) = |\cos(t)|$. 
Since 
$$
\int_0^\infty \gamma(t)\, dt = \int_0^\infty |\cos(t)|\, dt = \infty,
$$ 
the stability results in \cite{R4} and \cite{NH} are not applicable to this example as the results in \cite{R4} and \cite{NH} require $\sup_{t\ge 0}\int_0^t \gamma(\xi)\, d\xi < \infty$. 

{\bf Example 2.}  
Let $\mathcal{X} = \mathbb{R}^3$. 
Consider the equation
\begin{equation}
\label{eq94}
\dot{u} = B(t)u + \alpha(t)u^p + f(t),\qquad u(0)=u_0\in \mathbb{R}^3,
\end{equation}
where
$$
B(t): = 
\begin{bmatrix}
2\cos t - 1&0  &0 \\
0 & 2\cos(t-\frac{2\pi}{3}) - 1& 0\\
0 & 0 & 2\cos(t-\frac{4\pi}{3}) - 1\\
\end{bmatrix}
$$
and $f(t)$ and $u(t)$ are functions from $[0,\infty)$ to $\mathbb{R}^3$. 
Recall from \eqref{eqzx1} that $U(t,\xi) $ is the solution to
\begin{equation*}
\frac{\partial}{\partial t}U(t,\xi) =  
B(t)
U(t,\xi),\qquad t\ge \xi,\quad U(\xi,\xi) = I_3
\end{equation*}
where $I_3$ is the 3-by-3 identity matrix. 
The function $U(t,\xi) $ in this example is
\begin{equation*}
U(t,\xi) = \begin{bmatrix}
e^{\xi-t + 2\sin t  - 2\sin \xi} &0 &0\\
0 & e^{\xi-t + 2\sin(t-\frac{2\pi}{3}) - 2\sin(\xi -\frac{2\pi}{3})} &0\\
0 &   0&  e^{\xi -t + 2\sin(t-\frac{4\pi}{3}) - 2\sin(\xi -\frac{4\pi}{3})}\\
\end{bmatrix}.
\end{equation*}
Thus,
\begin{equation}
\label{bicho3}
\|U(t,\xi)\| \le e^{\xi -t + 4},\qquad \forall t\ge \xi\ge 0.
\end{equation}
By construction we have 
$$
U(t) = U(t,0)= \begin{bmatrix}
e^{-t + 2\sin t} &0 &0\\
0 & e^{-t + 2\sin(t-\frac{2\pi}{3}) + 2\sin(\frac{2\pi}{3})} &0\\
0 &   0&  e^{ -t + 2\sin(t-\frac{4\pi}{3}) + 2\sin(\frac{4\pi}{3})}\\
\end{bmatrix}
$$ 
and
$$
U^{-1}(t) = \begin{bmatrix}
e^{t - 2\sin t} &0 &0\\
0 & e^{t - 2\sin(t-\frac{2\pi}{3}) - 2\sin(\frac{2\pi}{3})} &0\\
0 &   0&  e^{ t - 2\sin(t-\frac{4\pi}{3}) - 2\sin(\frac{4\pi}{3})}\\
\end{bmatrix}.
$$ 
Therefore,
\begin{equation}
\label{geq118}
e^{-t - 4}<\|U(t)\|  < e^{-t + 4},\qquad \|[U(t)]^{-1}\|  < e^{t + 4},\qquad t\ge 0. 
\end{equation}
From inequality \eqref{bicho3} one gets
$$
\sup_{t\ge 0}\int_0^t \|U(t,\xi)\| \alpha(\xi)\, d\xi \le \sup_{t\ge 0}\int_0^t e^{\xi -t + 4} \alpha(\xi)\, d\xi
\le \|\alpha\|_\infty\sup_{t\ge 0}\int_0^t e^{\xi -t + 4}\, d\xi = e^4 \|\alpha\|_\infty,
$$
where $\|\alpha\|_\infty:=\sup_{t\ge 0}|\alpha(t)|$. 
Therefore, if $\|\alpha\|_\infty < \infty$, i.e., $\alpha(t)$ is bounded on $[0,\infty)$, then the inequalities in \eqref{ieq27} hold. Consequently, it follows from Theorem \ref{thm2.1} that the solution $u=0$ is Lyapunov stable under persistently acting perturbation $f(t)$. 

Let us discuss the application of Corollary \ref{corrollary2.5} to this example. 
From inequality \eqref{geq118} one gets
\begin{equation}
\label{geq119}
\int_0^\infty \alpha(\xi)\|U^{-1}(\xi)\| \|U(\xi)\|^p\, d\xi \le \int_0^\infty \alpha(\xi) e^{(1-p)\xi + 4(1+p)}\, d\xi.
\end{equation}
Assume that
\begin{equation}
\label{geq120}
\sup_{t\ge 0}\bigg(\frac{\beta(t)}{\alpha(t)}\bigg)^{\frac{1}{p}} e^{t+4} <  \bigg(\frac{1}{(p-1)\int_0^\infty \alpha(\xi)e^{(1-p)\xi + 4(1+p)}\, d\xi}\bigg)^{\frac{1}{p-1}}.
\end{equation}
From \eqref{geq118}, \eqref{geq120}, and \eqref{geq119} one gets
\begin{equation*}
\begin{split}
\sup_{t\ge 0}\bigg(\frac{\beta(t)}{\alpha(t)}\bigg)^{\frac{1}{p}} \frac{1}{\|U(t)\|} \le 
\sup_{t\ge 0}\bigg(\frac{\beta(t)}{\alpha(t)}\bigg)^{\frac{1}{p}} &e^{t+4} <\bigg(\frac{1}{(p-1)\int_0^\infty \alpha(\xi)e^{(1-p)\xi + 4(1+p)}\, d\xi}\bigg)^{\frac{1}{p-1}}\\
\le & \bigg(\frac{1}{(p-1)\int_0^\infty \alpha(\xi)\|U^{-1}(\xi)\| \|U(\xi)\|^p\, d\xi}\bigg)^{\frac{1}{p-1}}.
\end{split}
\end{equation*}
Thus, inequality \eqref{nov3} holds if inequality \eqref{geq120} holds. Also, inequality \eqref{eq2020x1} follows from inequality in \eqref{geq118}. 
Therefore, if 
$$
\|u_0\| \le \bigg(\frac{1}{(p-1)\int_0^\infty \alpha(\xi)\|U^{-1}(\xi)\| \|U(\xi)\|^p\, d\xi}\bigg)^{\frac{1}{p-1}} - 
\sup_{t\ge 0}\bigg(\frac{\beta(t)}{\alpha(t)}\bigg)^{\frac{1}{p}} \frac{1}{\|U(t)\|},
$$
then by Corollary \ref{corrollary2.5} one gets
$$
\|u(t)\| \le C_2 \|U(t)\| \le C_2 e^{-t + 4},\qquad t\ge 0. 
$$
This means that the solution $u(t)$ decays exponentially to 0 at the rate of $e^{-t}$. 

If inequality \eqref{eqs9} holds, then it follows from Theorem \ref{thm2.4} that 
\begin{equation}
\label{bicho1}
\|u(t)\| \le (\kappa + 1)\bigg[\|U(t)\|\|u_0\| + \int_0^t \|U(t,\xi)\|\beta(\xi)\,d\xi\bigg],\qquad t\ge 0. 
\end{equation}
We claim that if $\lim_{t\to\infty}\beta(t)=0$ then 
\begin{equation}
\label{bicho2}
\lim_{t\to\infty}\|u(t)\| = 0.
\end{equation}
To verify relation \eqref{bicho2} we will show that the right-hand side of inequality \eqref{bicho1} converges to 0 as $t$ to $\infty$. 
Since $\lim_{t\to\infty}\|U(t)\| = 0$ due to \eqref{geq118},
it suffices to show that 
\begin{equation*}
\lim_{t\to\infty} \int_0^t \|U(t)U^{-1}(\xi)\|\beta(\xi)\, d\xi =0.
\end{equation*}
It follows from inequality \eqref{bicho3} that
\begin{equation}
\label{no117}
\begin{split}
\lim_{t\to\infty} \int_0^t \|U(t)U^{-1}(\xi)\|\beta(\xi)\, d\xi  &\le 
\lim_{t\to\infty} e^4\int_0^t e^{\xi -t}\beta(\xi)\, d\xi 
=
e^4\lim_{t\to\infty} \frac{\int_0^t e^{\xi}\beta(\xi)\, d\xi }{e^t}=0.
\end{split}
\end{equation}
Indeed, if $\int_0^\infty e^{\xi}\beta(\xi)\, d\xi<\infty$, then it is clear that the last equality in \eqref{no117} holds. If $\int_0^\infty e^{\xi}\beta(\xi)\, d\xi=\infty$ and $\lim_{t\to\infty}\beta(t)=0$, then the last equality in \eqref{no117} follows from L'Hospital's rule. 

Let us show why the stability results in \cite{R4} and \cite{NH} are not applicable to equation \eqref{eq94}. 
Define 
\begin{equation*}
A = \begin{bmatrix}
2\cos t - 1&0  &0 \\
0 & 2\cos(t-\frac{2\pi}{3}) - 1& 0\\
0 & 0 & 2\cos(t-\frac{4\pi}{3}) - 1\\
\end{bmatrix}.
\end{equation*}
Then the norm of $A$ is $\gamma(t) := \|A\|= \max\{2\cos t, 2\cos(t-\frac{2\pi}{3}), 2\cos(t-\frac{4\pi}{3})\} - 1$. 
By its definition $\gamma(t)$ is a periodic function of period $2\pi$. In addition, we have
$$
\gamma(t) = \left \{ \begin{matrix}2\cos t - 1& \text{if} & -\pi/3 \le t \le \pi/3,\\ 2\cos(t-\frac{2\pi}{3}) -1 & \text{if} & \pi/3 \le t \le \pi,\\ 2\cos(t-\frac{4\pi}{3}) - 1& \text{if} & -\pi \le t \le -\pi/3.\end{matrix} \right.
$$
Thus, $\gamma(t) \ge 0$, $\forall t\ge 0$. 
Therefore, one has $\int_0^\infty \gamma(t) = \infty$ since $\gamma(t)$ is periodic and nonnegative on $[0,\infty)$. 
Recall that the stability results in \cite{R4} and \cite{NH} require $\gamma(t)$ to be integrable on $[0,\infty)$. 
Therefore, the stability results in \cite{R4} and \cite{NH} are not applicable to equation \eqref{eq94}. 

{\bf Example 3.} 
In this example we will demonstrate the application of the results in Section \ref{section3} to equations in Hilbert spaces with unbounded operators. 

Let $\mathcal{H}=L^2(\mathbb{R}^3)\oplus L^2(\mathbb{R}^3)$ the direct sum of $L^2(\mathbb{R}^3)$ and itself, $L=L^*$ be a selfadjoint operator in $L^2(\mathbb{R}^3)$ which is the closure of the Laplacian with domain of definition $C_0^\infty(\mathbb{R}^3)$. Consider the equation 
\begin{equation}
\label{no112}
\dot{u} = \begin{bmatrix}
\cos t &0 \\
0 & -\cos t
\end{bmatrix}
u
+
\begin{bmatrix}
L &0 \\
0 & L
\end{bmatrix}u
 + G(t,u) + f(t),\qquad u(0)=u_0.
\end{equation}
Assume that $\|G(t,u)\| \le \alpha(t)\|u\|^p$, $p>1$. There are many functions satisfying this inequality. For example, $G(t,u) = \alpha(t)u^p$. 

Equation \eqref{no112} is of the form \eqref{jq1} with 
$$
B(t) = \begin{bmatrix}\cos t & 0\\ 0 &-\cos t \end{bmatrix},\qquad A(t) = \begin{bmatrix}
L &0 \\
0 & L
\end{bmatrix}.
$$
The function $U(t)$, according to equation \eqref{q9}, is the solution to
\begin{equation*}
\frac{\partial}{\partial t}U(t) =  \begin{bmatrix}
\cos t &0 \\
0 & -\cos t
\end{bmatrix}
U(t),\qquad U(0) = I.
\end{equation*}
Thus,
\begin{equation}
\label{nov5.1}
U(t) = \begin{bmatrix}
e^{\sin t } &0 \\
0 & e^{-\sin t}
\end{bmatrix},\qquad U^{-1}(t) = \begin{bmatrix}
e^{-\sin t } &0 \\
0 & e^{\sin t}
\end{bmatrix}.
\end{equation}
It is clear that $\langle (U^{-1})^*U^{-1}Lu,u\rangle \le 0$ where $\langle \cdot,\cdot\rangle$ 
is the inner product in $\mathcal{H}$.  
From equation \eqref{nov5.1} one gets
\begin{equation}
\label{nov5.2}
\frac{1}{e}\le\|U(t)\| \le e,\qquad \frac{1}{e} \le \|U^{-1}(t)\| \le e,\qquad \forall t\ge 0.
\end{equation}
Therefore, 
\begin{equation}
\label{nov5.3}
\sup_{t\ge 0}\|U(t)\| \le e,\qquad \sup_{t\ge 0} \|U(t)\| \int_0^t \|U^{-1}(\xi)\| \alpha(\xi)\, d\xi \le e^2 \int_0^\infty \alpha(\xi)\, d\xi.
\end{equation}
Thus, if 
$$\int_0^\infty \alpha(t)\, dt <\infty$$ then it follows from\eqref{nov5.3} and Theorem \ref{thm3.1} that the solution $u=0$ is Lyapunov stable under persistently acting perturbation $f$.

{\it Let us demonstrate the application of Theorem \ref{thm3.3} to equation \eqref{no112}}. Assume that 
\begin{equation}
\label{nov5.4}
0<\|u_0\| < \frac{1}{\bigg[(p-1)\int_0^\infty e^{p+1} \alpha(\xi) \,d\xi \bigg]^{\frac{1}{p-1}}} - e\sup_{t\ge 0} \bigg(\frac{\beta(t)}{\alpha(t)}\bigg)^\frac{1}{p}.
\end{equation}
Then it follows from \eqref{nov5.2} and \eqref{nov5.4} that
\begin{equation}
\label{nov5.5}
\begin{split}
\sup_{t\ge 0} \bigg(\frac{\beta(t)}{\alpha(t)}\bigg)^\frac{1}{p} \frac{1}{\|U(t)\|} \le \sup_{t\ge 0}& \bigg(\frac{\beta(t)}{\alpha(t)}\bigg)^\frac{1}{p} e \le  \frac{1}{\bigg[(p-1)\int_0^\infty e^{p+1} \alpha(\xi) \,d\xi \bigg]^{\frac{1}{p-1}}}  - \|u_0\|\\
& \le \frac{1}{\bigg[(p-1)\int_0^\infty \|U^{-1}(\xi)\| \|U(\xi)\|^p\alpha(\xi) \,d\xi \bigg]^{\frac{1}{p-1}}}  - \|u_0\|.
\end{split}
\end{equation}
It follows from \eqref{nov5.3} and \eqref{nov5.5} that inequalities \eqref{no90} and \eqref{no91} hold. From inequalities \eqref{no90} and \eqref{no91} and Theorem \ref{thm3.3} one concludes that 
$\|u(t)\| \le C_2 \|U(t)\|$ for some constant $C_2>0$. 
This and the first inequality in \eqref{nov5.3} imply that the solution to equation \eqref{no112} is bounded if \eqref{nov5.4} holds. 

Now instead of having \eqref{nov5.4} we assume that
\begin{equation}
\label{jan9.1}
(\kappa + 1)^p\bigg(e\|u_0\| + e^2 \int_0^t \beta(\xi)\, d\xi\bigg)^p \le \kappa \frac{\beta(t)}{\alpha(t)},\qquad t\ge 0, \quad \kappa>0.
\end{equation}
It follows from \eqref{jan9.1} and \eqref{nov5.2} that
\begin{equation}
\label{jan9.2}
\begin{split}
\alpha(t) &\le \frac{\kappa\beta(t)}{(\kappa + 1)^p\big[e\|u_0\| + e^2 \int_0^t \beta(\xi)\, d\xi\big]^p}\\
&\le \frac{\kappa\beta(t)}{(\kappa + 1)^p\big[\|U(t)\|\|u_0\| + \|U(t)\| \int_0^t  \|U^{-1}(\xi)\| \beta(\xi)\, d\xi\big]^p},
\end{split}
\end{equation}
for all $t\ge 0$. It follows from \eqref{jan9.2} and Theorem \ref{thm3.4} that
$$
0\le \|u(t)\| \le (\kappa + 1)\bigg( \|U(t)\|\|u_0\| + \|U(t)\| \int_0^t  \|U^{-1}(\xi)\| \beta(\xi)\, d\xi\bigg),\qquad t\ge 0.
$$
This and \eqref{nov5.2} imply
$$
0\le \|u(t)\| \le (\kappa + 1)\bigg( e\|u_0\| + e^2 \int_0^t \beta(\xi)\, d\xi\bigg).
$$
Therefore, the solution $u(t)$ is bounded on $[0,\infty)$ if $\int_0^\infty \beta(\xi)\, d\xi < \infty$ and inequality \eqref{jan9.1} holds. Note that the stability results in \cite{R4} are also not applicable to this example. 

\section*{Conclusion}

The stability of solutions to abstract evolution equations has been studied under nonclassical assumptions. 
Our new results are applicable to equations in Banach spaces and equations in Hilbert spaces with unbounded operators under nonclassical assumptions. 
Using our new results, we have been able to obtain stability results for some simple equations to which the stability of the solutions cannot be obtained by using classical results in the literature.

\end{document}